%% file: tangle-tree_short160511.tex
\documentclass[11pt]{article}

\input{minimalJ}

\usepackage{verbatim}
\usepackage{enumerate}
\usepackage[all]{xy}

\usepackage{subfig}

\title{\scshape A short proof that every finite graph has a tree-decomposition displaying its tangles}

\author{Johannes Carmesin}

\newcommand{\ct}{^\complement}

\begin{document}

\maketitle

\begin{abstract}
We give a short proof that every finite graph (or matroid) has a tree-decomposition that displays all maximal tangles.

This theorem for graphs is a central result of the graph minors project of Robertson and Seymour 
and the extension to matroids is due to Geelen, Gerards and Whittle.
\end{abstract}

\section{Introduction}
Robertson and Seymour \cite{GMX} proved as a corner stone of their graph minors project:
\begin{thm}[rough version]\label{RS:tangle_tree}
Every graph\footnote{In this paper all graphs and matroids are finite.}
has a tree-decomposition whose separations distinguish all maximal tangles.
\end{thm}
Additionally, it can be ensured that this tree-decomposition separates the tangles in a `minimal way'.
This theorem was extended to matroids by Geelen, Gerards and Whittle \cite{GGW:tangles_in_matroids}.
Here we give a short proof of both of these results. A key idea is that we prove the following strengthening:

\begin{thm}[rough version of \autoref{tangle-tree_strong} below]\label{thm_intro}
Any tree-decomposition such that each of its 
separations distinguishes two tangles in a minimal way
can be extended to a tree-decomposition that distinguishes 
any two maximal tangles in a minimal way. 
\end{thm}

Our new proof does not yield the strengthening of \autoref{RS:tangle_tree} proved in \cite{CDHH:profiles}. 
However, it can be extended from tangles to profiles, compare \autoref{profile_rem}. 
For tree-decompositions as in \autoref{RS:tangle_tree} that additionally have as few parts as possible see 
\autoref{tangle_tree_cor}.

\section{Notation}

Throughout we fix a finite set $E$. A \emph{separation} is a bipartition $(A,B)$ of $E$, and $A$ and $B$ are called the \emph{sides} of $(A,B)$. 
A function $f$ mapping subsets of $E$ to the integers is \emph{symmetric} if $f(X)=f(X\ct)$ for every $X\se E$, and it is \emph{submodular} if $f(X)+f(Y)\geq f(X\cap Y) +f(X\cup Y)$ for every $X,Y\se E$. 
Throughout we fix such a function $f$. Since $f$ is symmetric, it induces a function $o$ on the separations: $o(A,B)=f(A)=f(B)$, which we call the \emph{order} of a separation.\footnote{For the sake of readability, we write $o(A,B)$ instead of $o((A,B))$.}  
Since $f$ is submodular $o$ satisfies:
\begin{equation}\label{lemma1_xyz}
  o(A,B)+o(C,D)\geq o(A\cap C, B\cup D)+ o(A\cup C, B\cap D)
\end{equation}

For example, one can take for $E$ the edge set of a matroid and for $f$ its connectivity function. Or one can take for $E$ the edge set of a graph, 
where the order of a separation $(A,B)$ is the number of vertices incident with edges from both $A$ and $B$.

A \emph{tangle} of order $k+1$ picks a \emph{small} side of each separation $(A,B)$ of order at most $k$
such that no three small sides cover $E$. 
Moreover, the complement of a single element of $E$ is never small.\footnote{This `moreover'-property is never used in our proofs and thus the results are also true for a slightly bigger class. However, the new objects are trivial.} 
In particular, if $A$ is small, then its complement $B$ cannot be included in a small set and we say that $B$ is \emph{big}. 
Thus a tangle can be thought of as pointing towards a highly connected piece, which `lies' on the big side of every low of order separation.
In this spirit, we shall also say that a tangle $\Tcal$ \emph{orients} a separation $(A,B)$ towards $B$ if $B$ is big in $\Tcal$.

A tangle is \emph{maximal} if it is not included in any other tangle (of higher order). 
A separation $(A,B)$ \emph{distinguishes} two tangles if these tangles pick different small sides for $(A,B)$.
It distinguishes them \emph{efficiently} if it has minimal order amongst all separations distinguishing these two tangles. 
\vspace{0.3 cm}

A \emph{tree-decomposition} consists of a tree $T$ and a partition $(P_t|t\in V(T))$ of $E$ consisting of one (possibly empty) partition class for every vertex of $T$. For $X\se V(T)$, we let $S(X)=\bigcup_{t\in X} P_t$.
There are two separations \emph{corresponding} to each edge $e$ of $T$, namely $(S(X), S(Y))$ and $(S(Y), S(X))$.
Here $X$ and $Y$ are the two components of $T-e$. 
We say that a tree-decomposition \emph{distinguishes two tangles efficiently} if there is a separation corresponding to an edge of the decomposition-tree distinguishing these tangles efficiently.

The following implies \autoref{RS:tangle_tree} and its matroid counterpart mentioned in the Introduction if we plug in the particular choices for the order function mentioned above.\footnote{In \cite{GMX}, the authors use a slightly different notion of separation for graphs.
From a separation $(A,B)$ in the sense of this paper, the corresponding separation in their setting is $(V(A), V(B))$, where $V(X)$ denotes the set of vertices incident with edges from $X$.
However, it is well-known that these two notions of separations give rise to the same notion of tangle and so \autoref{tangle_tree} implies their version. }
\begin{thm}\label{tangle_tree}
Let $E$ be a finite set with an order function.
Then there is a tree-decomposition distinguishing any two maximal tangles efficiently. 
\end{thm}

Two separations $(A_1,A_2)$ and $(B_1,B_2)$ are \emph{nested}\footnote{Other authors use \emph{laminar} instead.} if $A_i\se B_j$ for some pair $(i,j)\in \{1,2\}\times \{1,2\}$. 
A set of separations is \emph{nested} if any two separations in the set are nested. 
A set of separations $N$ is \emph{symmetric} if $(A,B)\in N$ if and only if $(B,A)\in N$. Note that any nested set $N$ is contained in a nested symmetric set, which consists of those separations $(A,B)$ such that $(A,B)$ or $(B,A)$ is in $N$. 
It is clear that:
\begin{rem}
Given a tree-decomposition, the set of separations corresponding to the edges of the decomposition-tree is nested and symmetric. 
\end{rem}

The converse is also true:
\begin{lem} \label{nested_to_td}[\cite{GGW:tangles_in_matroids}]
For every nested symmetric set $N$ of separations, there is a tree-decomposition such that the separations corresponding to edges of the decomposition-tree are precisely those in $N$.
\end{lem}

Hence to prove \autoref{tangle_tree}, it is enough to construct a suitable nested set of separations.
In the old proofs of \cite{GMX} or \cite{GGW:tangles_in_matroids}, the concept of robust separations was introduced in order to find such a set of separations.
We show that basically any nested set of separations works -- as long as it does not contain any useless separations and is maximal with this property:

\begin{thm}\label{tangle-tree_strong}
Let $N$ be any maximal nested set of separations such that each separation in $N$ distinguishes some two tangles efficiently.
Then any two maximal tangles are distinguished efficiently by some separation in $N$.
\end{thm}

Since \autoref{tangle-tree_strong} implies \autoref{tangle_tree}, the next section is dedicated to the proof of  \autoref{tangle-tree_strong}.

\section{Proof of \autoref{tangle-tree_strong}}

In our proof we need the following:
\begin{lem}\label{lemma2_xyz}\cite[Lemma 4.20]{C:undom_td_new}
 Let $(A,B)$, $(C,D)$ and $(E,F)$ be separations such that $(A,B)$ and $(C,D)$ are not nested but $(E,F)$ is nested with the other two separations. Then the corner separation $(A\cap C, B\cup D)$ is nested with $(E,F)$.
\end{lem}

\begin{proof}
Recall that if $(G,H)$ and $(E,F)$ are nested, then one of $G\se E$, $G\se E\ct$, $G\ct\se E$ or $G\ct\se E\ct$ is true.
If one of $G\se E$ or $G\se E\ct$ is false for $G=A\cap C$, then it is also false for both $G=A$ and $G=C$.
If one of $G\ct\se E$ or $G\ct\se E\ct$ is false for $G=A\cap C$, then it is false for at least one of $G=A$ or $G=C$.
Suppose for a contradiction that $(A\cap C, B\cup D)$ is not nested with $(E,F)$ but $(A,B)$ and $(C,D)$ are.
By exchanging the roles of $(A,B)$ and $(C,D)$ if necessary, we may assume by the above that  $A\ct\se E$ and $C\ct\se E\ct$.
Then $A\ct\se C$, contradicting the assumption that $(A,B)$ and $(C,D)$ are not nested.
\end{proof}

\begin{proof}[Proof of \autoref{tangle-tree_strong}.]
 Let $N$ be any maximal set of separations each distinguishing some two tangles efficiently.
Let $(A,B)$ be a separation distinguishing two maximal tangles $\cal P$ and $\cal Q$ efficiently.
Amongst all such $(A,B)$ we pick one such that the number of separations of $N$ not nested with $(A,B)$ is minimal.
By the maximality of $N$, it suffices to show that $(A,B)$ is nested with $N$.
By symmetry, it suffices to consider the case where $A$ is big in $\Pcal$ and $B$ is big in $\cal Q$. 

Suppose for a contradiction, there is some $(C,D)$ in $N$ not nested with $(A,B)$.
Let $\cal R$ and $\cal S$ be two maximal tangles distinguished efficiently by $(C,D)$ and without loss of generality $D$ is big in $\cal R$ and $C$ is big in $\Scal$.
Let $k$ be the order of $(A,B)$, and $\ell$ the order of $(C,D)$. 

\paragraph{Case 1: $k\geq \ell$.}

Then $\cal P$ and $\cal Q$ orient $(C,D)$. 
If they orient it differently, then $(C,D)$ is a candidate for $(A,B)$ and thus $(A,B)$ must be nested with $N$, which is the desired contradiction.
Since $N$ is maximal, it contains also $(D,C)$. 
Thus by replacing $(C,D)$ by $(D,C)$ if necessary, we may assume that $D$ is big in both $\cal P$ and $\cal Q$. 

Suppose for a contradiction that $(A\cap C, B\cup D)$ has order at least $\ell$.
Then $(A\cup C, B\cap D)$ has order at most $k$  by \autoref{lemma1_xyz}. 
Then $B\cap D$ is big in $\Qcal$ since three small sets cannot cover $E$ and both $B\ct$ and $D\ct$ are small.
On the other hand $B\cap D$ is small in $\Pcal$ since any subset of a small set cannot be big. 
However, by \autoref{lemma2_xyz} the separation $(A\cup C, B\cap D)$ is nested with every separation in $N$ that is nested with $(A,B)$ and additionally with $(C,D)$. This is a contradiction to the choice of $(A,B)$. 
Hence $(A\cap C, B\cup D)$ has order at most $\ell-1$.
By a similar argument $(B\cap C, A\cup D)$ has order at most $\ell-1$. 

The separation $(A\cap C, B\cup D)$ has a too low order to distinguish $\cal R$ and $\cal S$.
Since subsets of small sets cannot be big, $A\cap C$ is small in $\cal R$. Thus $A\cap C$ is also small $\cal S$.
A similar argument gives that $B\cap C$ is small in $\Scal$. 

But then $\Scal$ is not a tangle since its three small sets $D$, $A\cap C$ and $B\cap C$ cover $E$.
This is a contradiction.

\paragraph{Case 2: $k< \ell$.}

Then $\cal R$ and $\cal S$ orient $(A,B)$.
They cannot orient it differently as $(C,D)$ distinguishes them efficiently.
By replacing $(A,B)$ by $(B,A)$ if necessary, we may assume that $B$ is big in both $\cal R$ and $\cal S$. 

Suppose for a contradiction that $(A\cap C, B\cup D)$ has order at least $k+1$.
Then $(A\cup C, B\cap D)$ has order at most $\ell-1$ by \autoref{lemma1_xyz}. 
Then $B\cap D$ is big in $\Rcal$ since three small sets cannot cover $E$ and both $B\ct$ and $D\ct$ are small.
On the other hand $B\cap D$ is small in $\Scal$ since any subset of a small set cannot be big. 
Thus $(A\cup C, B\cap D)$ distinguishes $\Rcal$ and $\Scal$, which contradicts the efficiency of $(C,D)$. 
Hence $(A\cap C, B\cup D)$ has order at most $k$. 
A similar argument gives that $(A\cap D, B\cup C)$ has order at most $k$.

By  \autoref{lemma2_xyz}, these two corner separations are nested with every separation in $N$ that is nested with $(A,B)$ and additionally with $(C,D)$.
Thus by the choice of $(A,B)$, they cannot distinguish $\cal P$ and $\cal Q$.
Since subsets of small sets cannot be big, $A\cap C$ is small in $\cal Q$. So it is also small in $\cal P$.
Similarly, $A\cap D$ is small in $\cal P$.
But then $\Pcal$ is not a tangle since its three small sets $B$, $A\cap C$ and $A\cap D$ cover $E$.
This is a contradiction. 

\vspace{0.3 cm}

Since we derive a contradiction in both cases such a separation $(C,D)$ cannot exist and $(A,B)$ is nested with $N$. Thus since $N$ is maximal, for any two maximal tangles $\cal P$ and $\cal Q$, the set $N$ contains a separation distinguishing them efficiently.  
\end{proof}

\section{Concluding remarks}

\begin{rem}\label{profile_rem}\emph{
 \autoref{tangle-tree_strong} says that if we build a set of separations successively, where at each step we add a separation that is nested with everything 
 so far and distinguishes two tangles in a minimal way, then eventually we will end up with a nested set of separations that distinguishes any two maximal tangles in a minimal way.
 However, we could build our nested set of separations a little more carefully, taking smaller separations first. More precisely, our construction has a $k$-th subroutine for each $k\in \Nbb$ starting with $k=0$. 
At the $k$-th subroutine of our construction we add successively separations of order $k$ that are nested with everything so far and 
 that distinguish two tangles in a minimal way. We continue this until we can no longer proceed. 
  With basically the same proof as above (actually, since we take small separations first, we do not need to consider Case 2), one can show that any construction of this type does not only distinguish all the maximal tangles but more generally all the robust profiles as defined in \cite{CDHH:profiles}. 
 }
\end{rem}

\begin{rem}\label{algo_rem}\emph{
\autoref{tangle-tree_strong} gives rise to the following algorithm to construct a tree-decomposition 
that distinguishes all maximal tangles efficiently. At each step we have a nested set $N$ of 
separations such that each of its separations distinguishes some two tangles efficiently. 
If there are two maximal tangles that are not distinguished efficiently by a separation in $N$, 
our aim is to add some separation to $N$ that is nested with $N$ and distinguishes these 
two tangles efficiently.
\autoref{tangle-tree_strong} guarantees that this will always be possible no matter which 
choices we make on the way. 
 }
\end{rem}

Next we will define what it means for a tangle $\Qcal$ to live in a part of a tree-decomposition $(T,(P_t|t\in V(T)))$. 
If $tu$ is an edge of $T$ we abbreviate by $(S_t,S_u)$ the separation $(S(X_t),S(X_u))$, where $X_t$ is the component of $T-e$ containing $t$ and $X_u$ is the component of $T-e$ containing $u$. 
We say that $\Qcal$ \emph{lives} in a nonempty subgraph $S$ of $T$
if for every $t\in V(S)$ and every edge $tu$ incident with $t$ but not in $S$, the separation $(S_t,S_u)$ is big in $\Qcal$.
Clearly, every tangle $\Qcal$ lives in $T$ and the intersection of two subgraphs in which $\Qcal$ lives is nonempty and $\Qcal$ lives in that intersection.  
Hence there is a smallest subgraph $S(\Qcal)$ of $T$ in which $\Qcal$ lives. 
Clearly, if $\Qcal$ lives in $S$, then $S$ must be connected. So $S(\Qcal)$ is a tree. 
Also note that the order of a separation corresponding to an edge of $S(\Qcal)$ 
cannot be smaller than the order of $\Qcal$.
Hence if for two tangles $\Pcal$ and $\Qcal$, the sets $S(\Pcal)$ and $S(\Qcal)$ intersect, then no separation corresponding to an edge of $T$ distinguishes $\Pcal$ and $\Qcal$.  
We are mostly interested in the case where $S(\Qcal)$ just consists of a single node $t$. In this case we say that $\Qcal$ \emph{lives} in the part $P_t$. 
Our aim is to deduce the following.
\begin{cor}\label{tangle_tree_cor}
Let $E$ be a finite set with an order function.
Then there is a tree-decomposition distinguishing any two maximal tangles efficiently such that in each of its parts lives a maximal tangle. 
\end{cor}

By \autoref{tangle-tree_strong}, it suffices to show the following Lemma.
Given a nested set $N$ of separation, by $\Tcal(N)$ we denote the tree-decomposition of 
the smallest nested symmetric set containing $N$ in the sense of \autoref{nested_to_td}.

\begin{lem}\label{get_small_size}
Let $N$ be a nested set of separations that is minimal with the property that for any two maximal tangles there is a separation in $N$ that distinguishes them efficiently.
Then in each part of $\Tcal(N)$ lives a maximal tangle. Conversely, each maximal tangle lives in a part of $\Tcal(N)$. 
\end{lem}

\begin{proof}
By assumption the subtrees $S(\Qcal)$ for different tangles $\Qcal$ are disjoint. Hence the `conversely'-part follows from the first part.
Suppose for a contradiction, there is a part $P_t$ in which no tangle lives. Let $u$ be a neighbour of $t$ in $T$ such that 
the order of a separation $(A,B)$ corresponding to $tu$ is maximal. 

We will construct for any two tangles $\Pcal$ and $\Qcal$ distinguished efficiently by $(A,B)$ another separation in $N$ that also distinguishes them efficiently. 
Note that $tu$ separates $S(\Pcal)$ and $S(\Qcal)$. 
Since not both $S(\Pcal)$ and $S(\Qcal)$ can contain $t$, we may assume that $t$ is not in $S(\Pcal)$. 
Since $S(\Qcal)$ is not equal to $\{t\}$, there is a neighbour $r$ of $t$ that is different from $u$ such that the edge $tr$ separates a vertex of $S(\Qcal)$ from $S(\Pcal)$. Since the order of a separation corresponding to $tr$ is at most the order of a separation corresponding to $tu$, the node $t$ cannot be in $S(\Qcal)$. Hence a separation corresponding to  $tr$ distinguishes $\Pcal$ and $\Qcal$, and they do so efficiently as $(A,B)$ does. 

Hence $N-(A,B)$ violates the minimality of $N$, which contradicts our assumptions. Hence in each part of $\Tcal(N)$ lives a maximal tangle.
\end{proof}

\section{Acknowledgement}

I thank Reinhard Diestel for valuable comments concerning \autoref{algo_rem}.
\autoref{get_small_size} is a special case of a result in \cite{profiles_newVersion}, 
which proves the tangle-tree theorem in a more abstract setting. 

\vspace{.5cm}

\noindent Address of the author: \newline
University of Cambridge \newline
Wilberforce Road, Cambridge CB3 0WB \newline
E-Mail: j.carmesin@dpmms.cam.ac.uk

\bibliographystyle{plain}
\bibliography{literatur}

\end{document}

%% file: minimalJ.tex
\usepackage{epsfig}
\usepackage{graphicx}
\usepackage{amsbsy}
\usepackage{amsmath}
\usepackage{amsfonts}
\usepackage{amssymb}
\usepackage{textcomp}
\usepackage{hyperref}
\usepackage{aliascnt}

\newcommand{\mcm}[3]{\newcommand{#1}[#2]{{\ensuremath{#3}}}} 

\mcm{\tuple}{1}{\langle #1 \rangle}
\mcm{\name}{1}{\ulcorner #1 \urcorner}
\mcm{\Nbb}{0}{\mathbb{N}}
\mcm{\Zbb}{0}{\mathbb{Z}}
\mcm{\Rbb}{0}{\mathbb{R}}
\mcm{\Cbb}{0}{\mathbb{C}}
\mcm{\Qbb}{0}{\mathbb{Q}}
\mcm{\Bcal}{0}{\cal B}
\mcm{\Ccal}{0}{\cal C}
\mcm{\Dcal}{0}{\cal D}
\mcm{\Ecal}{0}{\cal E}
\mcm{\Fcal}{0}{\cal F}
\mcm{\Gcal}{0}{\cal G}
\mcm{\Hcal}{0}{\cal H}
\mcm{\Ical}{0}{\cal I}
\mcm{\Jcal}{0}{\cal J}
\mcm{\Kcal}{0}{\cal K}
\mcm{\Lcal}{0}{\cal L}
\mcm{\Mcal}{0}{\cal M}
\mcm{\Ncal}{0}{\cal N}
\mcm{\Ocal}{0}{{\cal O}}
\mcm{\Pcal}{0}{{\cal P}}
\mcm{\Qcal}{0}{{\cal Q}}
\mcm{\Rcal}{0}{{\cal R}}
\mcm{\Scal}{0}{{\cal S}}
\mcm{\Tcal}{0}{{\cal T}}
\mcm{\Ucal}{0}{{\cal U}}
\mcm{\Vcal}{0}{{\cal V}}
\mcm{\Wcal}{0}{{\cal W}}
\mcm{\Xcal}{0}{{\cal X}}
\mcm{\Ycal}{0}{{\cal Y}}
\mcm{\Mfrak}{0}{\mathfrak M}

\mcm{\restric}{0}{\upharpoonright}
\mcm{\upset}{0}{\uparrow}
\mcm{\onto}{0}{\twoheadrightarrow}
\mcm{\smallNbb}{0}{{\small \mathbb{N}}}
\DeclareMathOperator{\preop}{op}
\mcm{\op}{0}{^{\preop}}

\newcommand{\se}{\subseteq}
{\begin{array}{c}
\setlength{\unitlength}{1em}}%
{\end{array}}

\usepackage{amsthm}

\newcommand{\theoremize}[2]{\newaliascnt{#1}{thm} \newtheorem{#1}[#1]{#2} \aliascntresetthe{#1}}

\theoremstyle{plain}
\newtheorem{thm}{Theorem}[section]
\theoremize{lem}{Lemma}
\theoremize{sublem}{Sublemma}
\theoremize{claim}{Claim}
\theoremize{rem}{Remark}
\theoremize{prop}{Proposition}
\theoremize{cor}{Corollary}
\theoremize{que}{Question}
\theoremize{oque}{Open Question}
\theoremize{con}{Conjecture}

\theoremstyle{definition}
\theoremize{dfn}{Definition}
\theoremize{eg}{Example}
\theoremize{exercise}{Exercise}
\theoremstyle{plain}

%% file: tangle-tree_short160511.bbl
\begin{thebibliography}{1}

\bibitem{C:undom_td_new}
J.~Carmesin.
\newblock All graphs have tree-decompositions displaying their topological
  ends.
\newblock Preprint 2014, available at {http://arxiv.org/pdf/1409.6640v4}.

\bibitem{CDHH:profiles}
J.~Carmesin, R.~Diestel, M.~Hamann, and F.~Hundertmark.
\newblock Canonical tree decompositions of finite graphs {I}: Existence and
  algorithms.
\newblock {\em J.~Combin.\ Theory (Series B)}, pages 1--24, 2016.

\bibitem{GGW:tangles_in_matroids}
Jim Geelen, Bert Gerards, and Geoff Whittle.
\newblock Tangles, tree-decompositions and grids in matroids.
\newblock {\em J. Combin. Theory Ser. B}, 99(4):657--667, 2009.

\bibitem{GMX}
N.~Robertson and P.D. Seymour.
\newblock Graph minors. {X}. {O}bstructions to tree-decompositions.
\newblock {\em J.~Combin.\ Theory (Series B)}, 52:153--190, 1991.

\end{thebibliography}
